\documentclass[11pt]{amsart}
\usepackage{amsmath, amsfonts, amsthm, amssymb, color, tikz, multicol, enumitem}

\numberwithin{equation}{section}

\theoremstyle{plain}
\newtheorem{theo}{Theorem}
\newtheorem{thm}[theo]{Theorem}
\newtheorem{lem}[theo]{Lemma}
\newtheorem{prop}[theo]{Proposition}
\newtheorem{cor}[theo]{Corollary}

\theoremstyle{plain}
\newtheorem{nontheo}{nonTheorem}
\newtheorem{ques}[nontheo]{Question}

\newcommand{\Z}{\mathbb{Z}}
\newcommand{\R}{\mathbb{R}}

\newcommand{\lattice}{\Lambda}
\newcommand{\rect}{\mathcal{R}}
\newcommand{\M}{\mathcal{M}}
\newcommand{\sphere}{\mathcal{S}}
\newcommand{\nsphere}{\mathcal{S}^{n-1}}
\newcommand{\area}{\text{area}}

\newcommand{\nball}{\mathcal{B}_{n}}
\newcommand{\A}{\mathcal{A}}
\newcommand{\norm}[1]{\| #1 \|}
\renewcommand\bar[1]{\overline{#1}}
\newcommand{\diag}{\textnormal{diag}}
\renewcommand{\S}{\mathcal{S}}

\title{Solid angles associated to Minkowski~reduced~bases}

\makeatletter
\@namedef{subjclassname@2010}{\textup{2010} Mathematics Subject Classification}
\makeatother

\author{Danny Nguyen}
\address{Department of Mathematics,  
University of California, Los Angeles, 
520 Portola Plaza,
Los Angeles, CA 90095}
\email{ldnguyen@math.ucla.edu}

\begin{document}

\begin{abstract}
Given a lattice $\lattice \subset \R^n$, we consider its Minkowski reduced basis and the solid angle $\Omega$ spanned by the basis vectors. Such a basis satisfies strong near-orthogonality conditions, which allow us to bound  from above and below the measure of $\Omega$. Sharp upper and lower bounds are derived for all rank $3$ and rank $4$ lattices so that $\Omega$ always measures in between. Extreme cases happen when $\lattice$ is similar to the rectangular ($\mathcal{R}$) or alternating ($\A$) lattice. This result settles a question raised earlier by Fukshansky and Robins in connection to sphere packings and kissing numbers. The proof relies on a formula by Hajja and Walker that expresses $\Omega$ as a product of $\det(\lattice)$ and a quadratic integral on the unit sphere $\sphere^{n-1}$. Finally, we show that for rank 5, the alternating lattice $\A_{5}$ no longer possesses the smallest measure for $\Omega$. 
\end{abstract}

\maketitle

\section{Review of the problem}\label{intro}

A rank $n$ lattice $\lattice$ is a discrete subgroup of $\R^n$ generated by $n$ linearly independent vectors. Such a set of vectors is a basis of $\lattice$ under integer linear combinations. As in the case of general groups, the choice of the basis is non-unique. However, since $\lattice$ carries the Euclidean metric, a reduction process can be applied to yield a "shortest" or \emph{minimal basis}. There have been many reduction processes devised for this purpose, most notably Minkowski, Korkine-Zolotarev and LLL reductions. In this paper, by a minimal or reduced basis we always mean one that results from Minkowski's reduction process. We describe this simple reduction process below.
\medskip

A set of $n$ vectors $\{v_1,\dots,v_n\}$ form a Minkowski reduced basis if $v_1$ is the shortest non-zero vector in $\Lambda$ and for each $1<k\leq n$, $v_k$ is the shortest vector that makes $v_1,\dots,v_k$ is \textit{extendable} to a full basis of $\Lambda$. To put in another way, $\{ v_1,\dots,v_n \}$ must generate $\Lambda$ by integer linear combinations, and if $(x_1,\dots,x_n) \in \Z^n$ is any $n$-tuple satisfying $\gcd(x_k,\dots,x_n)=1$ for some $1 \leq k \leq n$, then $\| v_k \| \leq \| \sum{x_i v_i} \|$. Note that for each $v_i$, there are more than one shortest vector available ($-v_i$ for example). So when we refer to a minimal basis, we mean one among many available minimal bases. As a special case, when $\Lambda$ has a basis consisting of orthogonal vectors, it is automatically a reduced basis. In the general case, a reduced basis is the closest to an orthogonal basis that a lattice can have.
\medskip

Motivated by various extremal geometric problems including sphere packings and kissing numbers, for which the optimal solution often involves a minimal basis of some special lattice, Fukshansky and Robins \cite{FukRob} posed a question on finding sharp bounds for the \textit{solid angles} associated to such minimal bases. Given a minimal basis of $n$ vectors $v_1,\dots,v_n$, they generate a cone $K = \{x_1 v_1 + \dots + x_{n} v_{n} : x_{i} \in \R^{+} \}$ and the associated solid angle $\Omega$ is defined as:
\[
\Omega = \area(K \cap \nsphere),
\]
where $\area(\cdot)$ denotes the $(n-1)$-dim spherical area on $\nsphere$. The question raised in \cite{FukRob} is:

\begin{ques}\label{mainques}
Find absolute constants $C_{1}$ and $C_{2}$ so that every rank $3$ lattice has a minimal basis with the associated solid angle $\Omega$ satisfying $C_{1} \le \Omega \le C_{2}$?
\end{ques}

In dimension $3$, Fukshansky and Robins employed L'huilier's to express the solid angle $\Omega$ as:

\small
\begin{equation}\label{lhuilier}
\tan\left(\frac{\Omega}{4}\right)^2=\tan\left(\frac{\alpha+\beta+\gamma}{2}\right) \tan\left(\frac{\alpha+\beta-\gamma}{2}\right) \tan\left(\frac{\beta+\gamma-\alpha}{2}\right) \tan\left(\frac{\gamma+\alpha-\beta}{2}\right).
\end{equation}
\normalsize


\begin{figure}[!h]
\centering
\begin{tikzpicture}[scale=1]
\draw[->] (0,0) -- (3,0);
\draw[->,dashed] (0,0) -- (2.5,1.5);
\draw[->] (0,0) -- (1,3);

\coordinate (A3) at (1.4,0.85);
\coordinate (B3) at (1.15,0);
\coordinate (C3) at (0.45,1.4);

\fill[color=lightgray] (A3) to[bend left=20] (B3) to[bend right=17] (C3) to[bend left=22] (A3);

\node[right] (1) at (3,0) {$v_1$};
\node[above right] (2) at (2.5,1.5) {$v_2$};
\node[above] (3) at (1,3) {$v_3$};

\node (4) at (1.15,0.8) {\small$\Omega$};

\end{tikzpicture}
\caption{Solid angle in dimension 3.}
\label{fig:lhuilier}
\end{figure}
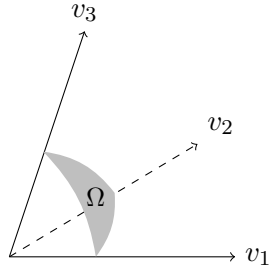

\noindent Here $\alpha,\beta$ and $\gamma$ are the pairwise angles between the three basis vectors $v_{1}, v_{2}, v_{3}$. As we will see later on, they satisfy $\frac{\pi}{3} \leq \alpha,\beta,\gamma \leq \frac{2\pi}{3}$ whenever the basis is minimal. With some other extra assumptions on $\alpha,\beta,\gamma$, Fukshansky and Robins proved that 

\medskip

\begin{thm}[Corollary 3.3 in \cite{FukRob}]\label{FRthm}
For a wide class of rank $3$ lattices including the well-rounded lattices, a minimal basis exists with $\Omega$ satisfying:
\[
\tan\left(\frac{\pi}{12}\right)^3 \le \tan\left(\frac{\Omega}{4}\right)^2 \le \tan\left(\frac{\pi}{8}\right)^2.
\]
The upper and lower bounds are sharp, and they attained when $\lattice$ is similar\footnote{Two lattices are similar if one can be obtained from the other by an orthogonal transformation followed by a scalar multiplication.} to respectively the rectangular lattice $\mathcal{R}_3$ and the alternating lattice $\A_3$. $\mathcal{R}_3$ is generated by $\{ (1,0,0), (0,1,0), (0,0,1) \}$. $\A_3$ is also called the rank $3$ face-centered cubic lattice, and is generated by $\{ (\frac{1}{\sqrt{2}},\frac{1}{\sqrt{2}},0), (\frac{1}{\sqrt{2}},0,\frac{1}{\sqrt{2}}), (0,\frac{1}{\sqrt{2}},\frac{1}{\sqrt{2}}) \}$.
\end{thm}

The above theorem covers the important case of \emph{well-rounded} (WR) lattices, i.e., those with a minimal basis consisiting of equal length vectors. These lattices are important in discrete optimization and also give good solutions to the \emph{kissing number problem} (see \cite{FukRob}, \cite{ConSlo}). Some technical condition however prevents the proof to apply to all rank $3$ lattices. Furthermore, an analogue of \eqref{lhuilier} is not known in dimensions higher than $3$. In this paper, we give a complete answer to Question~\ref{mainques} in dimensions $3$ and $4$:

\medskip

\begin{thm}\label{mainthm}
For any rank $3$ lattice, a minimal basis exists with solid angle $\Omega$ satisfying:
\[
\Omega_{\A_3} \le \Omega \le \Omega_{\mathcal{R}_3},
\]
where $\Omega_{\A_3} = 0.551285\ldots, \Omega_{\mathcal{R}_3} = \frac{\pi}{2}$ are absolute constants\footnote{These constants matches with those in Theorem~\ref{FRthm}.}. Similarly, for any rank $4$ lattice, a minimal basis exists with solid angle $\Omega$ satisfying:
\[
\Omega_{\A_4} \le \Omega \le \Omega_{\mathcal{R}_4},
\]
where $\Omega_{\A_4} = 0.193142\ldots, \Omega_{\mathcal{R}_4} = \frac{\pi^2}{8}$. 
These bounds are sharp. $\Omega_{\mathcal{R}_3}$ and $\Omega_{\mathcal{R}_4}$ are attained when $\lattice$ is similar to respectively $\rect_{3}$ and $\rect_{4}$; $\Omega_{\A_3}$ and $\Omega_{\A_4}$ are attained when $\lattice$ is similar to the alternating lattices $\A_{3}$ and $\A_{4}$.
\end{thm}

In the next section, we outline the differences between the methods of proof for Theorem~\ref{FRthm} (\cite{FukRob}) and our Theorem~\ref{mainthm}. We will give necessary and sufficient conditions to check that a basis is reduced, and also describe formula to compute $\Omega$ given the basis in matrix form. 

\medskip

\section{Conditions for a reduced basis and a formula for solid angles}\label{background}

We first mention the related concept of \emph{successive minima} for a lattice. Given a full rank lattice $\Lambda \subset \R^n$, its successive minima 
\[
0 < \lambda_1 \le \dots \le \lambda_n
\]
are defined as:
\[
\lambda_i = \inf \{\lambda \in \R^{+} : \Lambda \cap \lambda \nball \text{ contains } i \text{ linearly independent vectors} \},
\]
where $\nball$ is the unit ball in $\R^n$. Associated to each $\lambda_i$ is a vector $u_i \in \Lambda$ with $\norm{u_{i}} = \lambda_i$. Even though we only require $u_1, \dots, u_n$ to be linearly independent, $\{u_1, \dots , u_{n}\}$ actually forms a minimal basis for $\Lambda$ when $n \le 4$. This is an important fact, whose proof can be found in standard texts on Geometry of Numbers (see \cite{Sie}, \cite{GruLek}). However, for $n \ge 5$, this does not always hold. Using this fact, Fukshansky and Robins treated a reduced basis in dimension $3$ as successive minima vectors. They showed that:

\begin{lem}[Lemma 2.3 in \cite{FukRob}]\label{anglecond}
Let $\Lambda \in \R^n$ be a full-rank lattice with successive minima $0 < \lambda_1 \le \dots \le \lambda_n$ and associated lattice vectors $u_{1}, \dots, u_{n}$, chosen so that they all lie in the same half-space. Then for every pair $u_{i}, u_{j}$ $(1 \le i < j \le n)$, the angle $\theta_{ij}$ between them satisfies
\[
\frac{\pi}{3} \le \theta_{ij} \le \frac{2\pi}{3}.
\]
\end{lem}

This implies that in dimension $3$, the angles $\alpha, \beta$ and $\gamma$ in L'huilier's formula \eqref{lhuilier} are in between $\frac{\pi}{3}$ and $\frac{2\pi}{3}$. This was a crucial ingredient for the proof of Theorem~\ref{FRthm}. Nevertheless, Lemma~\ref{anglecond} is not a sufficient condition for $\{u_{1}, \dots, u_{n}\}$ to be a reduced basis, since it totally leaves out the relations between the lengths and pairwise angles of the $n$ basis vectors. 
\medskip

To give necessary and sufficient conditions for a minimal basis, we refer back to the original definition given at the beginning of Section 1. The collection $\{v_{1}, \dots, v_{n}\}$ is a reduced basis for $\Lambda$ iff they can generate $\Lambda$ under integer linear combinations, and for any $n$-tuple $(x_{1}, \dots, x_{n}) \in \Z^{n}$ satisfying $\gcd(x_{k}, \dots, x_{n}) = 1$ for some $1 \le k \le n$, we have $\norm{v_{k}} \le \norm{\sum{x_{i} v_{i}}}$. This characterization at the outset requires an infinite number of inequalities but, but Minkowski proved a result saying that we only need to check a finite number of inequalities involving dot products between the basis vectors. This is most conveniently expressed in terms of the Gram matrix associated to the basis. Call $A$ the $n\times n$ matrix having $v_i$'s as columns, then the Gram matrix $Q = A^{t}A$ has entries $q_{ij}=q_{ji}=\langle v_i,v_j \rangle$. $Q$ is positive definite and $\det(Q)={\det(A)}^2$ is the squared volume of the fundamental parallelepiped with edges $v_1,\dots,v_n$. The \textit{Minkowski reduction conditions} are linear inequalities in the $q_{ij}$'s, satisfying which $Q$ would be called \textit{reduced}.
\medskip

\begin{thm}[Minkowski, see \cite{Sie}, \cite{GruLek}]\label{Minthm}
Given $n$ linearly independent vectors $v_{1}, \dots, v_{n}$ in $\R^{n}$. Let $\Lambda$ be the lattice they generate and $Q$ be the Gram matrix with $q_{ij} = \langle v_{i}, v_{j} \rangle$. Then $\{v_{1}, \dots, v_{n}\}$ is a reduced basis for $\Lambda$ iff the entries $q_{ij}$'s satisfy a fixed set of linear inequalities, which only depend on the dimension $n$. 
\end{thm}

Abusing the language, we call any symmetric matrix $Q$ satisfying such inequalities a \emph{reduced form}. Reduction in $\R^2$ is particularly simple and was known to Gauss. In this case, $Q = \left(\begin{smallmatrix} a&b \\ b&c \end{smallmatrix}\right)$ is reduced exactly when 
\[
a\leq c \quad \text{and} \quad 2|b| \leq a.
\]
These correspond to the inequalities
\begin{equation}\label{2dimcond}
\| v_1 \| \leq \| v_2 \| \quad \text{and} \quad 2|\langle v_1,v_2 \rangle| \leq {\|v_1\|}^2.
\end{equation}
A more geometric way to look at the second inequality is 
\[
\norm{v_2} \leq \|v_1 - v_2\| \quad \text{and} \quad \norm{v_{2}} \le \|v_1 + v_2\|.
\]
Together with $\norm{v_1} \le \norm{v_{2}}$, these are exactly the finite collection of inequalities that Minkowski's theorem refers to. An important corollary, which the reader can verify, is that \eqref{2dimcond} implies 
\[
\frac{| \langle v_1,v_2 \rangle |}{\| v_1 \| \| v_2 \|} \leq \frac{1}{2},
\]
which means $v_1$ is separated from $v_1$ by an angle at least $\frac{\pi}{3}$ and at most $\frac{2\pi}{3}$. Thus, we can recover the necessary condition in Lemma~\ref{anglecond}.
\medskip

The reduction conditions get more involved as the dimension increases. The case $n=3$ requires 9 inequalities. Namely for $Q = \left( \begin{smallmatrix} a&d&e \\[0.22em] d&b&f \\[0em] e&f&c \end{smallmatrix} \right)$ to be reduced, we must have:
\begin{itemize}
\item[3a)] $a\leq b\leq c$.
\item[3b)] $2|d|\leq a$;  $2|e|\leq a$;  $2|f| \leq b$.
\item[3c)] $a+b+2(d+e+f) \geq 0$;  $a+b+2(d-e-f) \geq 0$;\\  $a+b+2(e-d-f) \geq 0$;  $a+b+2(f-d-e) \geq 0$.
\end{itemize}
For a proof of this and also the general theorem of Minkowski, please refer to \cite{Sie} and \cite{GruLek}. Reduction conditions for $n=4$ are even more involved. 
\medskip

Coming now to evaluating the solid angle, the following formula proved by Hajja and Walker \cite{HajWal} expresses the solid angle $\Omega$ in terms of $\det(Q)$ and the associated quadratic form $x^{t}Qx$. The formula is:
\begin{equation}\label{HWform}
\Omega_Q = \sqrt{\det(Q)} \int_{\sphere} {(x^{t}Qx)^{-n/2}}ds.
\end{equation}
Here $\sphere$ is the part of $\sphere^{n-1}$ lying in the positive orthant and $ds$ is the element of surface area on $\sphere^{n-1}$. The \emph{normalized} solid angle $\bar{\Omega}_{Q}$ is defined as:
\begin{equation*}\label{HWformnorm}
\bar{\Omega}_{Q} = \frac{\Omega_{Q}}{S_{n-1}},
\end{equation*}
where $S_{n-1}=\textnormal{area}(\sphere^{n-1})=\displaystyle\frac{n\pi^\frac{n}{2}}{\Gamma\left(\frac{n}{2}+1\right)}$. In low dimensions, $\Omega_Q$ is largely influenced by $\det(Q)$, whereas in higher dimensions the influence is weaker. This is explained by the phenomenon that most of the unit ball's volume concentrates near to its boundary in higher dimensions. This poses serious difficulties in optimizing \eqref{HWform} with $q_{ij}$'s as variables when $n$ is large, since we have to analyze the integral part more carefully. However, when $n \leq 4$, we can still manage to find the extrema for $\Omega_Q$ by first optimizing at $\det(Q)$. For our proof, we will fix the diagonal elements of $Q$ and then minimize $\det(Q)$, keeping the condition that $Q$ is reduced.
\medskip

The next section will carry out this minimization process for $\det(Q)$ in $\R^3$ and $\R^4$. Section~\ref{Secdim3} settles the bounds for $\Omega_Q$ for all rank 3 lattices. Section~\ref{Secdim4} deals with rank 4 lattices by the same method, but more work will be required. Finally in Section~\ref{Secdim5}, we give an example showing that the alternating lattice $\A_5$ no longer has the smallest solid angle in dimension 5 and discuss some open questions.

\medskip

\section{Minimizing the determinant}\label{Secdet}

A general method was described in the work of Barnes (\cite{Bar1}, \cite{Bar2}), which allows us to find the exact minimum of $\det(Q)$, with the conditions that $\diag(Q)$ is fixed and $Q$ must remain a reduced form. First, let us recall the definition of \emph{quasi-concavity}. A real function $f$ on $\R^k$ is quasi-concave if 
\[
f(\lambda x + (1-\lambda)y) \geq \text{min}(f(x),f(y)) \text{ for any } x,y \in \R^k \text{ and }0\leq\lambda\leq 1. 
\]
We first prove a simple but useful fact:

\medskip

\begin{lem}
The determinant function is quasi-concave on the restricted domain of symmetric positive definite matrices.
\end{lem}
\begin{proof}
It is equivalent to show that if $\det(Q_2) \geq \det(Q_0) \geq \alpha > 0$ then $Q = \lambda Q_0 + (1-\lambda)Q_2$ has $\det(Q) \geq \alpha$. We can write $Q_0 = O^{t}DO$, with $O$ an orthogonal matrix and $D$ a diagonal matrix with all positive diagonal entries. Let $E = \sqrt{D}$ and $K=EO$, we have $Q_0=K^{t}K$. Now 
\[
Q = \lambda Q_0 + (1-\lambda) Q_2 = K^{t}(\lambda I + (1-\lambda)K^{-t}{Q_2}K^{-1})K.
\] 
Let $H=K^{-t}{Q_2}K^{-1}$, we have:
\[
\det(Q) = \det(Q_0)\det(\lambda I + (1-\lambda)H) \geq \alpha \det(\lambda I + (1-\lambda)H).
\]

Note that $H$ is also symmetric and $\det(H)=\frac{\det(Q_2)}{\det(Q_0)} \geq 1$. Therefore $\lambda I + (1-\lambda)H$ is diagonalizable and 
\[
\det(\lambda I + (1-\lambda)H) = \prod (\lambda + (1-\lambda)h_i)
\]
with $h_i$ being the eigenvalues of $H$. Using AM-GM inequality, we have 
\[
\lambda + (1-\lambda)h_i \geq {h_i}^{1-\lambda}
\]
Hence $\det(Q) \geq \alpha (\prod{h_i})^{1-\lambda} = \alpha(\det(H))^{1-\lambda} \geq \alpha$.

\end{proof}

Another way to look at quasi-concavity of $f$ is that $R_{\alpha} = \{x \in \R^k : f(x) \geq \alpha\}$ is always a convex set for any $\alpha \in \R$. Minkowski's theorem says that a reduced Gram matrix $Q$ must satisfy certain linear inequalities that depend on the dimension $n$. These inequalities correspond to certain half-spaces in the space of all symmetric $n\times n$ matrices, and their intersection is a polyhedral cone. We call this cone $\M_n$. If we fix $\diag(Q)$, then we are restricted to the intersection of $\M_n$ with $n$ hyperplanes, and intersection becomes a convex polytope $\M_n^\diag$. By quasi-concavity of the determinant function, we know that the minima for $\det(Q)$ are located among the polytope's vertices. These vertices can be found explicitly by taking all possible intersections of any $\frac{n(n-1)}{2}$ different facets of $\M_{n}^{\diag}$, and check whether they actually belong to $\M_n^{\diag}$. For an easy illustration, when $n=2$, $Q = \left(\begin{smallmatrix} a&b \\ b&c \end{smallmatrix}\right)$, the hyperplanes defining $\M_2$ are 
\[
a\leq c, -2b\leq a \; \text{ and } \; 2b\leq a.
\]
Fixing $a$ and $c$, we see that the polytope $\M_{n}^{\diag}$ is just a line segment with two vertices $\{(a,-\frac{a}{2},c),(a,\frac{a}{2},c)\}$ and the minimum determinant is $\left(ac - \frac{a^2}{4}\right)$. It was further shown in \cite{Bar1} and \cite{Bar2} that:

\medskip

\begin{theo}[\cite{Bar1}, \cite{Bar2}]\label{Barlemma}
Fixing $\diag(Q)$ of the reduced Gram matrix $Q$, we have:
\begin{itemize} 
\item[\textnormal{a)}] If $n=3$ and $\diag(Q) = [a,b,c]$, then $\det(Q)\geq \frac{abc}{2}+\frac{ab(c-b)}{4}+\frac{ac(b-a)}{4}$, with the minimum achieved at three different forms.
\item[\textnormal{b)}] If $n=4$ and $\diag(Q) = [a,b,c,d]$, then $\det(Q)\geq \frac{abcd}{4}+\frac{acd(b-a)}{4}+\frac{abd(c-b)}{4}+\frac{abc(d-c)}{4}+\frac{a^2 (b-c)^2}{16}$ with the minimum achieved at fourteen different forms.
\end{itemize}
\end{theo}

\noindent The method of proof as mentioned above is to find all vertices of $\M_{n}^{\diag}$ and compare the values of $\det(Q)$ at those points. The explicit three/fourteen forms with minimum determinant are given in \cite{Bar1}, \cite{Bar2}. From now on, we are using square brackets to list the diagonal and upper-diagonal elements of a symmetric matrix. For instance, $Q = \left( \begin{smallmatrix} a&d&e \\[0.22em] d&b&f \\[0em] e&f&c \end{smallmatrix} \right)$ is encoded as $Q=[a, d, e; \; b, f; \; c]$. We now prove two technical lemmas which will be used later in Section~\ref{Secdim3} and \ref{Secdim4}. The reader can skip them for the moment. Let us assume that $a_1,a_2,a_3,b_1,b_2,c_1$ in the next two lemmas are real numbers satisfying:

\begin{itemize}
\item[\textnormal{i)}] $0\leq a_1,a_2,a_3,b_1,b_2,c_1 \leq \frac{1}{2}$
\item[\textnormal{ii)}] $3+2(a_1+c_1)-2(a_2+b_2+a_3+b_1) \geq 0;$ \\$3+2(a_2+b_2)-2(a_1+c_1+a_3+b_1) \geq 0;$ \\ $ 3+2(a_3+b_1)-2(a_1+c_1+a_2+b_2) \geq 0$.
\end{itemize}

\begin{lem}\label{tech1}
Fixing $c_1$, the determinant of $Q=[1,a_1,a_2,a_3; \; 1,b_1,b_2; \; 1,c_1; \; 1]$ is minimized when $a_1 = \frac{1}{2}-c_1$ and $a_2=a_3=b_1=b_2=\frac{1}{2}$.
\end{lem}
\begin{proof}
Fixing $c_1$ along with conditions i) and ii) means that $\M_{4}^{\diag}$ is a 5-dimensional convex polytope. Here we find all quintuples $(a_1,a_2,a_3,b_1,b_2)$ that correspond to the vertices of $\M_{4}^{\diag}$. Some of these however are equivalent because of the symmetry between $(a_2,b_2)$ and $(a_3,b_1)$, and therefore will yield the same value for $\det(Q)$. Below we list one vertex for each different $\det(Q)$ value, and the corresponding formula for $\det(Q)$:

\begin{multicols}{2}
\begin{itemize}[leftmargin=0.2in]
\item[]$(0,0,0,0,0): 1-{c_1}^2$
\item[]$(0,\frac{1}{2},\frac{1}{2},0,0): \frac{1}{2}+\frac{c_1}{2}-{c_1}^2$
\item[]$(0,0,\frac{1}{2},\frac{1}{2},0): \frac{9}{16}-{c_1}^2$
\item[]$(\frac{1}{2},0,0,0,0): \frac{3}{4}-\frac{3}{4}{c_1}^2$
\item[]$(\frac{1}{2},\frac{1}{2},\frac{1}{2},0,0): \frac{1}{4}+\frac{c_1}{2}-\frac{3}{4}{c_1}^2$
\item[]$(\frac{1}{2},0,\frac{1}{2},\frac{1}{2},\frac{1}{2}): \frac{5}{16}+\frac{c_1}{4}-\frac{3}{4}{c_1}^2$
\item[]$(\frac{1}{2},\frac{1}{2},0,0,\frac{1}{2}-c_1): \frac{5}{16}+\frac{c_1}{2}-{c_1}^2$
\item[]$(0,\frac{1}{2},0,0,0): \frac{3}{4}-{c_1}^2$
\item[]$(0,0,\frac{1}{2},0,\frac{1}{2}): \frac{1}{2}-{c_1}^2$
\item[]$(0,0,\frac{1}{2},\frac{1}{2},\frac{1}{2}): \frac{5}{16}+\frac{c_1}{2}-{c_1}^2$
\item[]$(\frac{1}{2},\frac{1}{2},0,0,0): \frac{1}{2}-\frac{3}{4}{c_1}^2$
\item[]$(\frac{1}{2},0,\frac{1}{2},0,\frac{1}{2}): \frac{1}{2} - \frac{3}{4}{c_1}^2$
\item[]$(\frac{1}{2},\frac{1}{2},\frac{1}{2},\frac{1}{2},\frac{1}{2}): \frac{1}{4}+\frac{c_1}{2}-\frac{3}{4}{c_1}^2$
\item[]$(\frac{1}{2},\frac{1}{2},0,c_1,\frac{1}{2}): \frac{5}{16}+\frac{c_1}{4}-\frac{3}{4}{c_1}^2$
\item[]$(0,c_1,\frac{1}{2},\frac{1}{2},\frac{1}{2}): \frac{5}{16}+\frac{c_1}{4}-\frac{3}{4}{c_1}^2$
\item[]$(\frac{1}{2}-c_1,\frac{1}{2},\frac{1}{2},\frac{1}{2},\frac{1}{2}): \left(\frac{1}{2}+\frac{c_1}{2}-{c_1}^2\right)^2$
\item[]$(\frac{1}{2}-c_1,\frac{1}{2},0,0,\frac{1}{2}): \frac{5}{16}+\frac{3}{4}c_1-\frac{5}{4}{c_1}^2-{c_1}^3+{c_1}^4.$
\end{itemize}
\end{multicols}

\medskip

\noindent It is tedious but straightforward to verify that the vertex $(\frac{1}{2}-c_1,\frac{1}{2},\frac{1}{2},\frac{1}{2},\frac{1}{2})$ has smallest determinant for all $c_1 \in (0,\frac{1}{2})$, and therefore the corresponding form $Q=[1,\frac{1}{2}-c_1,\frac{1}{2},\frac{1}{2};\; 1,\frac{1}{2},\frac{1}{2};\; 1,c_1;\; 1]$.
\end{proof}

\medskip

\begin{lem}\label{tech2}
~
\begin{itemize}
\item[\textnormal{a)}] Fixing $c_1 \leq \frac{1}{4}$, the determinant of $Q=[1,\frac{1}{2},a_2,a_3;\; 1,b_1,b_2;\; 1,c_1;\; 1]$ is minimized when $a_2=a_3=b_1=b_2=\frac{1}{2}$.
\item[\textnormal{b)}] If $c_1 > \frac{1}{4}$, we have
\[
\det([1,\frac{1}{2},a_2,a_3;\; 1,b_1,b_2;\; 1,c_1;\; 1]) > \det([1,\frac{1}{2},a_2,a_3;\; 1,b_1,b_2;\; 1,\frac{1}{2};\; 1]).
\]
\end{itemize}
\end{lem}
\begin{proof}
a) Similar to the previous lemma, we look at the vertices of the polytope containing all quadruples $( a_2,a_3,b_1,b_2)$. Now since $a_1 = \frac{1}{2}$, the first inequality in condition ii) holds automatically and the remaining conditions are 
\[
0 \leq a_2,a_3,b_1,b_2 \leq \frac{1}{2}
\]
\[
1-c_1+(a_3+b_1)-(a_2+b_2) \geq 0 \; \text{ and } \;  1-c_1+(a_2+b_2)-(a_3+b_1)\geq 0.
\]
Below we list one vertex for each equivalence class and the corresponding determinant:
\begin{multicols}{2}
\begin{itemize}[leftmargin=0.4in]
\item[]$( 0,0,0,0 ): \frac{3}{4} - \frac{3}{4}{c_1}^2$
\item[]$( \frac{1}{2},0,0,0 ): \frac{1}{2}-\frac{3}{4}{c_1}^2$
\item[]$( \frac{1}{2},\frac{1}{2},0,0 ): \frac{1}{4}+\frac{c_1}{2}-\frac{3}{4}{c_1}^2$
\item[]$ ( 0,\frac{1}{2},0,\frac{1}{2} ): \frac{1}{2}-\frac{3}{4}{c_1}^2$
\item[]$( 0,\frac{1}{2},\frac{1}{2},\frac{1}{2} ): \frac{5}{16}+\frac{c_1}{4}-\frac{3}{4}{c_1}^2$
\item[]$( \frac{1}{2},\frac{1}{2},\frac{1}{2},\frac{1}{2} ): \frac{1}{4}+\frac{c_1}{2}-\frac{3{x_1}^2}{4}$
\item[]$( \frac{1}{2},0,c_1,\frac{1}{2} ): \frac{5}{16}+\frac{c_1}{4}-\frac{3}{4}{c_1}^2$
\item[]$(\frac{1}{2},0,0,\frac{1}{2}-c_1 ): \frac{5}{16}+\frac{c_1}{2}-{c_1}^2$
\end{itemize}
\end{multicols}

\noindent By direct comparison for $c_1 \in [0,\frac{1}{4}]$, we see that $Q = [1,\frac{1}{2},\frac{1}{2},\frac{1}{2};1,\frac{1}{2},\frac{1}{2};1,c_1;1]$ has the smallest determinant.
\\

\noindent b) We have:
\begin{align*}
&\det([1,\tfrac{1}{2},a_2,a_3;\; 1,b_1,b_2;\; 1,c_1;\; 1]) - \det([1,\tfrac{1}{2},a_2,a_3;\; 1,b_1,b_2;\; 1,\tfrac{1}{2};\; 1])\\[0.4em]
&= (c_1-\tfrac{1}{2})(2 a_2 a_3 + 2 b_1 b_2 - a_3 b_1 - a_2 b_2 - \tfrac{3}{4}(c_1+\tfrac{1}{2})).
\end{align*}
Here we have $c_1-\frac{1}{2} \leq 0$ and also:
\[
2 a_2 a_3 + 2 b_1 b_2 - a_3 b_1 - a_2 b_2 = a_2 a_3 + b_1 b_2 + (a_2 - b_1)(a_3 - b_2).
\]
If $(a_2 - b_1)(a_3 - b_2) < 0$, then $a_2 a_3 + b_1 b_2 + (a_2 - b_1)(a_3 - b_2) < a_2 a_3 + b_1 b_2 \leq \frac{1}{2}$. Otherwise, we can assume that $a_2 \geq b_1$ and $a_3 \geq b_2$, and we have: 
\begin{align*}
a_2 a_3 + b_1 b_2 + (a_2 - b_1)(a_3 - b_2) &\leq \tfrac{1}{4} + b_1 b_2 + (\tfrac{1}{2}-b_1)(\tfrac{1}{2}-b_2) \\[0.4em]
&= \tfrac{1}{2}+\tfrac{1}{2}(4 b_1 b_2 - b_1 - b_2) \\[0.4em]
&\leq \tfrac{1}{2} + \tfrac{1}{2}(4\tfrac{1}{2} \min\{ b_1,b_2 \} - b_1 - b_2) \leq \tfrac{1}{2}.
\end{align*} 
In any case, we have 
\[
2 a_2 a_3 + 2 b_1 b_2 - a_3 b_1 - a_2 b_2 - \frac{3}{4}(c_1+\frac{1}{2}) \leq \frac{1}{2} - \frac{3}{4}(\frac{1}{4}+\frac{1}{2}) < 0.
\]
So the conclusion is $(c_1-\frac{1}{2})(2 a_2 a_3 + 2 b_1 b_2 - a_3 b_1 - a_2 b_2 - \frac{3}{4}(c_1+\frac{1}{2})) \geq 0$.
\end{proof}

\medskip

\section{The 3-dimensional case}\label{Secdim3}

In this section we establish the sharp bounds in Theorem~\ref{mainthm} for rank 3 lattices. Let us consider the formula
\begin{equation}\label{normsolidform}
\bar{\Omega}_Q = \frac{\Omega_{Q}}{S_{n-1}} = \frac{\sqrt{\det(Q)}}{S_{n-1}} \int_{\sphere} {(x^{t}Qx)^{-n/2}}ds.
\end{equation}
A notable feature of the above formula is that if we replace $x_1$ by $\alpha x_1$ in $x = (x_{1}, \dots, x_{n})$, then the value of $\int_{\sphere} {(x^{t}Qx)^{-n/2}}ds$ is scaled by a factor $\frac{1}{\alpha}$. Doing so is tantamount to scaling the first basis vector $v_1$ by $\alpha$, which also scales $\sqrt{\det(Q)}$ by $\alpha$. Since $\Omega_{Q}$ remains the same even if we scale one of the basis vectors, the value of the integral must be scaled by $\frac{1}{\alpha}$. We first prove a minor result.

\begin{prop}\label{minor}
If $Q$ has all positive entries then $\bar{\Omega}_Q \leq \displaystyle\frac{1}{2^n}$.
\end{prop}
\begin{proof}
Call $q_{11},q_{22},\dots,q_{nn}$ the diagonal entries of $Q$ then by Hadamard's inequality for positive definite matrix, we have $\det(Q) \leq \prod q_{ii}$. Also because of the assumption on positivity of all entries, we have $x^{t}Qx \geq \sum {q_{ii}{x_i}^2}$. Hence
\begin{align*}
\bar{\Omega}_Q&\leq \frac{\sqrt{\prod q_{ii}}}{S_{n-1}} \int_{\sphere}{(\sum {q_{ii}{x_i}^2})^{-n/2}}ds \\
&= \frac{1}{S_{n-1}} \int_{\sphere}{(\sum {x_i}^2)^{-n/2}}ds = \frac{1}{2^n}
\end{align*}
\end{proof}

\begin{theo}\label{A3min}
Any reduced basis of any rank $3$ lattice has $\Omega_Q\geq \Omega_{\A_3}$ with $\A_3$ the rank $3$ face-centered cubic lattice generated by $(\frac{1}{\sqrt{2}},\frac{1}{\sqrt{2}},0), (\frac{1}{\sqrt{2}},0,\frac{1}{\sqrt{2}}), (0,\frac{1}{\sqrt{2}},\frac{1}{\sqrt{2}})$.
\end{theo}
\begin{proof}
Let $Q = [a,a_{1},a_{2};\; b,b_{1};\; c]$. Recall the conditions 3a-3c. By 3a and Lemma~\ref{tech1}a), we have $\sqrt{\det(Q)}\geq \sqrt{\frac{abc}{2}}$. Also, replacing $x_1$ with $\frac{\sqrt{c}}{\sqrt{a}}x_1$ and $x_2$ with $\frac{\sqrt{c}}{\sqrt{b}}x_2$, we get:
\begin{align*}
&\quad\;\int_{\sphere}{(x^{t}Qx)^{-\frac{3}{2}}}ds \\ &=\int_{\sphere}{(a{x_1}^2+b{x_2}^2+c{x_3}^2+2{a_1}{x_1}{x_2}+2{a_2}{x_1}{x_3}+2{b_1}{x_2}{x_3})^{-\frac{3}{2}}}ds \\
&=\sqrt{\frac{c^2}{ab}}\int_{\sphere}{(c({x_1}^2+{x_2}^2+{x_3}^2)+\tfrac{2{a_1}c}{\sqrt{ab}}{x_1}{x_2}+\tfrac{2{a_2}\sqrt{c}}{\sqrt{a}}{x_1}{x_3}+\tfrac{2{b_1}\sqrt{c}}{\sqrt{b}}{x_2}{x_3})^{-\frac{3}{2}}}ds
\end{align*}
From the reduction conditions 3a and 3b, we have $\max(|a_1|, |a_{2}|) \leq \frac{a}{2}, |b_{1}| \le \frac{b}{2}$ and $a\leq b\leq c$. These imply $\frac{2a_1c}{\sqrt{ab}}, \frac{2{a_2}\sqrt{c}}{\sqrt{a}}, \frac{2{b_1}\sqrt{c}}{\sqrt{b}} \leq c$. We have
$$\int_{\sphere}{(x^{t}Qx)^{-\frac{3}{2}}}ds \geq \frac{1}{\sqrt{abc}}\int_{\sphere}{({x_1}^2+{x_2}^2+{x_3}^2+{x_1}{x_2}+{x_1}{x_3}+{x_2}{x_3})^{-\frac{3}{2}}}ds$$
From these two bounds for $\sqrt{\det(Q)}$ and $\int_{\sphere}{(x^{t}Qx)^{-\frac{3}{2}}}ds$ we get
$$\Omega_Q \geq \frac{1}{\sqrt{2}} \int_{\sphere}{({x_1}^2+{x_2}^2+{x_3}^2+{x_1}{x_2}+{x_1}{x_3}+{x_2}{x_3})^{-\frac{3}{2}}}ds = \Omega_{\A_3}.$$
\end{proof}

\begin{cor}
Any rank $3$ lattice has a reduced basis with $ \Omega_{\A_3} \leq \Omega_Q \leq \frac{1}{8} S_{2} = \frac{\pi}{2}$. 
\end{cor}

\begin{proof}
Pick a reduced basis and switch basis vectors to their negatives if necessary to ensure that $\bar{\Omega}_Q \leq \frac{1}{8}$ (the three basis vectors together with their negatives give us eight cones to choose from). By the above theorem, we also have the lower bound.  
\end{proof}

This proves the rank 3 case of Theorem~\ref{mainthm}. It should be noticed that the quadratic form $Q_{\A_3}=[1,\frac{1}{2},\frac{1}{2};\; 1,\frac{1}{2};\; 1]$ that minimizes $\Omega_{Q}$ lies on the boundary of $\M_3$. This fact also persists in higher dimensions:

\medskip

\begin{theo}\label{minonboundary}
If $Q \in \M_n$ has the smallest solid angle $\Omega_Q$, then $Q$ must lie on $\partial(\M_n)$, which is a union of facets of $\M_n$ corresponding to the reduction inequalities.
\end{theo}
\begin{proof}
With a quick reference to the explicit reduction conditions 3a)-3c) for $\M_3$ listed in Section~\ref{background}, the two inequalities in 3a) mean the basis vectors are arranged in increasing norms, we call these as \textit{first-type} reduction conditions. The other conditions in 3b) and 3c) are of \textit{second-type}. When $Q$ achieves the minimum for $\Omega_Q$, we can actually say something stronger. Namely, for any vector $v_i$, at least one of the second-type reduction conditions must attain equality, which involves some coefficient $q_{ij}$ with $i\neq j$. Consider $v_1$ for instance, if all the second-type reduction conditions containing some $q_{1j}$ are strict, we can change $v_1$ to ${v_1}'$ that lies within the 2-dimensional angle between $v_1$ and $v_2$. We can take ${v_1}'$ to have the same length with $v_1$, and the angle between ${v_1}'$ and $v_2$ slightly smaller than that between $v_1$ and $v_2$. This means $q_{11}$ is kept constant but $q_{1j}$ will be slightly changed, and all the reduction conditions still hold. Moreover, ${v_1}'$ is now a positive linear combination of $v_1$ and $v_2$, and therefore the cone $K'$ generated by $\{v_{1}', v_{2}, \dots, v_{n}\} $ is contained inside the original cone $K$ generated by $\{v_{1},\dots,v_{n}\}$. So $K'$ has a smaller solid angle measure compared to $K$. This would contradict the assumption on $\Omega_Q$'s minimality.

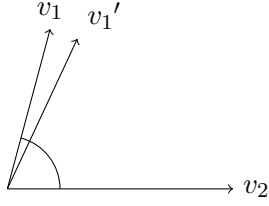
\begin{figure}[!h]
\centering
\begin{tikzpicture}[scale=1]
\draw[->] (0,0) -- (3,0) node[right] {$v_2$};
\draw[->] (0,0) -- (75:2.2) node[above] {$v_1$};
\draw[->] (0,0) -- (65:2.2) node[above right] {${v_1}'$};
\draw (0.7,0) arc (0:75:0.7);
\end{tikzpicture}
\caption{Slightly rotating $v_1$ will reduce $\Omega$.}
\label{fig:jiggle}
\end{figure}
\end{proof}

\medskip

\section{The 4-dimensional case}\label{Secdim4}

In this section we prove the rank 4 case of Theorem~\ref{mainthm}. Our proof strategy is to narrow down the search from the domain $\M_4$ to only forms with non-negative entries. From there, we will further narrow down to WR forms, which have all diagonal entries equal $1$. These steps will significantly simplify the complicated reduction conditions in $\R^4$. First, it is necessary to mention here the exact reduction conditions in $\R^4$, which were used by Barnes to prove Theorem~\ref{Barlemma}. It was shown in \cite{BarCoh} that a positive definite
\[
Q = \begin{pmatrix} q_{11} & q_{12} & q_{13} & q_{14} \\ q_{12} & q_{22} & q_{23} & q_{24} \\ q_{13} & q_{23} & q_{33} & q_{34} \\ q_{14} & q_{24} & q_{34} & q_{44} \end{pmatrix}
\]
is reduced when:

\medskip

\begin{itemize}
\item[4a)] $q_{11}\leq q_{22}\leq q_{33}\leq q_{44}$.
\item[4b)] For each $1 < i \le 4$, we must have $x^{t}Qx \geq q_{ii}$ for any $x=(x_1,x_2,x_3,x_4)$ satisfying $x_i=1$, $x_j=0$ if $j>i$, $x_j \in \{0,1,-1\}$ if $j<i$, and $x_j\neq 0$ for at least one $j<i$.
\end{itemize}

\medskip

The 36 second-type inequalities in 4b) consist of 28 inequalities which we already met in 3b)-3c). Those in fact tell us that the four rank 3 sublattices generated by $\{v_2,v_3,v_4\}$, $\{v_1,v_3,v_4\}$, $\{v_1,v_2,v_4\}$ and $\{v_1,v_2,v_3\}$ are also reduced. The other eight inequalities are added to compare $\|v_4\|$ with $\| \pm v_1 \pm v_2 \pm v_3 + v_4 \|$. Indexing the entries $q_{ij}$ row-by-column makes it easy to summarize all 39 reduction conditions, but from now on, we label the entries of $Q$ as: 
\[
Q = \begin{pmatrix} a & a_1 & a_2 & a_3 \\ a_{1} & b & b_1 & b_2 \\ a_{2} & b_{1} & c & c_1 \\ a_{3} & b_{2} & c_{1} & d \end{pmatrix}.
\]

We will prove that under these conditions
\[
Q_{\A_{4}}=[1,\frac{1}{2},\frac{1}{2},\frac{1}{2};\; 1,\frac{1}{2},\frac{1}{2};\; 1,\frac{1}{2};\; 1],
\]
the analogue of $Q_{\A_3}$, has the smallest solid angle $\Omega_{\A_4}$. Even though this is the case, $Q_{\A_4}$ no longer has the smallest determinant among all reduced WR forms. That property now belongs to 
\[
Q_0=[1,0,\frac{1}{2},\frac{1}{2};\; 1,\frac{1}{2},\frac{1}{2};\; 1,\frac{1}{2};\; 1].
\] 
Here the single $0$ can actually take any off-diagonal position. In fact, 
\[
\det(Q_0) = \frac{1}{4} < \det(Q_{\A_4}) = \frac{5}{16}.
\]
However, $Q_{\A_4}$ has the the largest possible values for off-diagonal entries and that helps minimize the integral $\int_{\sphere} {(x^{t}{Q_{\A_4}}x)^{-2}}ds$. At the end, we will compare $\Omega_{Q_{\A_4}}$ to	 $\Omega_{Q_0}$ numerically but it can be first proved that $\Omega_{Q_0}$ is smaller than a large class of solid angles.

\medskip

\begin{theo}\label{negisbad}
If $Q$ has any non-positive off-diagonal entry then $\Omega_Q \geq \Omega_{Q_0}$.
\end{theo}
\begin{proof}
This goes similar to the proof of Theorem~\ref{A3min}. By the condition 4a) and Theorem~\ref{Barlemma}b), we have
\[
\sqrt{\det(Q)}\geq \frac{\sqrt{abcd}}{2}.
\]
Let $M = \diag[\frac{\sqrt{d}}{\sqrt{a}}, \frac{\sqrt{d}}{\sqrt{b}}, \frac{\sqrt{d}}{\sqrt{c}}, 1]$ be a diagonal matrix and $Q' = M Q M$. By the property mentioned at the beginning of Section~\ref{Secdim3}, we have $\int_{\S} (x^{t} Q' x)^{-2} ds = \sqrt{\frac{abc}{d^3}} \int_{\S} (x^{t} Q x)^{-2} ds$. 
From \eqref{normsolidform}, we have:
\begin{align*}
\Omega_Q &= \sqrt{\det{Q}} \int_{\S} (x^{t} Q x)^{-2} dx \\
&\geq \frac{\sqrt{abcd}}{2} \sqrt{\frac{d^3}{abc}} \int_{\sphere} {(x^{t}Q'x)^{-2}}ds=\frac{d^2}{2}\int_{\sphere} {(x^{t}Q'x)^{-2}}ds.
\end{align*}
The new Gram matrix $Q'$ has all diagonal entries equal to $d$, each off-diagonal entry is at most $\frac{d}{2}$ and more importantly one such entry, say $a_1$, is non-positive. Therefore:
$$(x^{t}Q'x)^2 \leq d^2\left({x_1}{x_3} + {x_1}{x_4} + {x_2}{x_3} + {x_2}{x_4} + {x_3}{x_4} + \sum_{i=1}^{4}{x_i}^2\right)^2 = (d x^t Q_0 x)^2.$$ 
And so:
\[
\Omega_Q \geq \frac{d^2}{2} \int_{\sphere}{(d x^{t}{Q_0}x)^{-2}}ds \ge \frac{1}{2} \int_{\sphere}{( x^{t}{Q_0}x)^{-2}}ds = \Omega_{Q_0}.
\]
\end{proof}

By this result, we can narrow down our search to forms with all non-negative entries. This significantly reduces the number reduction conditions. It can be easily checked that all the reduction conditions similar to those in 3a)-3c) are now satisfied, and also all five vectors 
\begin{gather*}
\{(v_1+v_2+v_3+v_4),(-v_1-v_2-v_3+v_4),\\
(-v_1+v_2+v_3+v_4),(v_1-v_2+v_3+v_4),(v_1+v_2-v_3+v_4)\} 
\end{gather*}
have norm at least that of $v_4$. So there are 12 remaining reduction conditions and we rearrange them as:

\smallskip

\begin{itemize}
\item[4a)] $a\leq b\leq c\leq d$; $0 \leq a_i \leq \frac{a}{2}$; $0 \leq b_i \leq \frac{b}{2}$; $0 \leq c_i \leq \frac{c}{2}$.
\item[4b)] $(a+b+c)+2(a_1+c_1)-2(a_2+b_2+a_3+b_1) \geq 0$;\\
$(a+b+c)+2(a_2+b_2)-2(a_1+c_1+a_3+b_1) \geq 0$; \\
$(a+b+c)+2(a_3+b_1)-2(a_1+c_1+a_2+b_2) \geq 0$.
\end{itemize}

\smallskip

It should be noticed that in the last three inequalities, the 6 off-diagonal entries are now grouped into three pairs $(a_1,c_1)$, $(a_2,b_2)$ and $(a_3,b_1)$. This observation is important for many results following afterwards.

$$Q = \begin{pmatrix}
a & {\textcolor{green}{a_1}} & {\textcolor{red}{a_2}} & {\textcolor{blue}{a_3}} \\
{\textcolor{green}{a_1}} & b & {\textcolor{blue}{b_1}} & {\textcolor{red}{b_2}} \\
{\textcolor{red}{a_2}} & {\textcolor{blue}{b_1}} & c & {\textcolor{green}{c_1}} \\
{\textcolor{blue}{a_3}} & {\textcolor{red}{b_2}} & {\textcolor{green}{c_1}} & d
\end{pmatrix} $$

\medskip

\begin{lem}\label{WRgood}
In $\R^4$, the minimal solid angle is attained among WR forms.
\end{lem}
\begin{proof}
Given a reduced basis with Gram matrix $Q$ satisfying 4a)-4b), we rescale the basis vectors to be of equal length and prove that the resulting WR form  $Q$ is still reduced. This leaves the solid angle $\Omega_Q$ unchanged. First, scale $v_4$ by a factor of $\frac{\sqrt{c}}{\sqrt{d}}$. Thus $d \rightarrow c$ and 
\[
\left( a_1,a_2,a_3,b_1,b_2,c_1 \right) \to \left( a_1,a_2,\sqrt{\frac{c}{d}}a_3,b_1,\sqrt{\frac{c}{d}}b_2,\sqrt{\frac{c}{d}}c_1 \right).
\]
This decreases the magnitude of $c_1,b_2,a_3$ and so the inequalities in 4a) still hold. 

We need to verify that the inequalities in 4b) still hold. For the first inequality in 4b), since $a+b+c+2a_1 - 2a_2 - 2b_1 \geq 0$, if $2c_1 - 2a_3 - 2b_2 \geq 0$ then 
\[
a+b+c + (2a_1 + 2\sqrt{\frac{c}{d}}c_1) - (2a_2 + 2\sqrt{\frac{c}{d}}b_2 + 2\sqrt{\frac{c}{d}}a_3 + 2b_1) \geq 0.
\]
Otherwise, assume $2c_1 - 2a_3 - 2b_2 < 0$. We also have 
\[
1 \ge \sqrt{\frac{c}{d}}  \text{ and }  \sqrt{\frac{c}{d}}(2c_1 - 2a_3 - 2b_2) \ge 2c_1 - 2a_3 - 2b_2.
\] 
And so 
\begin{gather*}
a+b+c + (2a_1 + 2\sqrt{\frac{c}{d}}c_1) - (2a_2 + 2\sqrt{\frac{c}{d}}b_2 + 2\sqrt{\frac{c}{d}}a_3 + 2b_1) \geq \\
\geq a+b+c+2(a_1+c_1)-2(a_2+b_2+a_3+b_1) \geq 0.
\end{gather*}
We see that the first inequality in 4b) holds in any case. Similar arguments can verify the other two inequalities.

Now we can assume that $d=c$ and $Q = [a,a_1,a_2,a_3; \; b,b_1,b_2; \; c,c_1; \; c]$ satisfies conditions 4a)-4b). Next, scale $v_1$ by a factor of $\sqrt{\frac{b}{a}}$ so that $a \rightarrow b$ and 
\[
(a_1,a_2,a_3,b_1,b_2,c_1) \rightarrow \left( \sqrt{\tfrac{b}{a}}a_1,\sqrt{\tfrac{b}{a}}a_2,\sqrt{\tfrac{b}{a}}a_3, b_1, b_2, c_1 \right). 
\]
Since $a_i \leq \frac{a}{2}$ and $a \leq b$, we have$\sqrt{\frac{b}{a}}a_i \leq \frac{b}{2}$ and so 4a) still holds. For the first inequality in 4b), first note that:
\begin{align*}
&( b - 2\sqrt{\tfrac{b}{a}}a_2 - 2\sqrt{\tfrac{b}{a}}a_3 ) - (a - 2a_2 - 2a_3) \\
&= a(\tfrac{b}{a}-1) - 2a_2 ( \sqrt{\tfrac{b}{a}}-1 ) - 2a_3 ( \sqrt{\tfrac{b}{a}}-1 ) \\
&= (\sqrt{\tfrac{b}{a}}-1)(a(\sqrt{\tfrac{b}{a}}+1) - 2a_2 - 2a_3) \\
&\geq (\sqrt{\tfrac{b}{a}}-1)(2a - 2a_2 - 2a_3) \geq 0,
\end{align*}
i.e., $(b - 2\sqrt{\tfrac{b}{a}}a_2 - 2\sqrt{\tfrac{b}{a}}a_3) \ge (a - 2a_2 - 2a_3)$. Since also $\sqrt{\frac{b}{a}}a_1 \geq a_1$, we have 
\begin{align*}
&(b+b+c) + 2 ( \sqrt{\tfrac{b}{a}}a_1 + c_1 ) - 2 ( \sqrt{\tfrac{b}{a}}a_2 + b_2 + \sqrt{\tfrac{b}{a}}a_3 + b_1 ) \\
= \; &(b+c) + ( b - 2\sqrt{\tfrac{b}{a}}a_2 - 2\sqrt{\tfrac{b}{a}}a_3 )  + 2\sqrt{\tfrac{b}{a}}a_1 + 2c_{1} - 2b_{2} - 2b_{1} \\
= \; &(b+c) + ( a - 2a_{2} - 2a_{3} )  + 2a_{1} + 2c_{1} - 2b_{2} - 2b_{1} \\
\geq \; &(a+b+c) + 2(a_1 + c_1) - 2(a_2 + b_2 + a_3 + b_1) \geq 0. 
\end{align*}
So the first inequality in 4b) still holds. We can verify the other two equalities of 4b) in a similar manner and confirm that $Q$ is still reduced. 

So now we can assume $a=b$, $c=d$ and $Q = [b, a_1,a_2,a_3; \; b,b_1,b_2; \; c,c_1; \; c]$ satisfies conditions 4a)-4b). The last step is scaling both $v_1$ and $v_2$ up by a factor of $\sqrt{\frac{c}{b}}$. Hence $b \rightarrow c$, $a_1 \rightarrow \frac{c}{b}a_1$ and $(a_2,a_3,b_1,b_2) \rightarrow (\sqrt{\frac{c}{b}}a_2,\sqrt{\frac{c}{b}}a_3,\sqrt{\frac{c}{b}}b_1,\sqrt{\frac{c}{b}}b_2)$. Observe that all the off-diagonal entries do not decrease in magnitude. Like the previous steps, we can easily prove that $\frac{c}{b}a_1,\sqrt{\frac{c}{b}}a_2,\sqrt{\frac{c}{b}}a_3,\sqrt{\frac{c}{b}}b_1,\sqrt{\frac{c}{b}}b_2 \leq \frac{c}{2}$. This means 4a) holds for the resulting WR from. It is not hard to prove that:
\begin{align*}
c - \sqrt{\tfrac{c}{b}}(a_2 + a_3 + b_1 + b_2) &\geq b - (a_2+a_3+b_1+b_2) \\
c - \tfrac{c}{b}a_1 - \sqrt{\tfrac{c}{b}}a_3 - \sqrt{\tfrac{c}{b}}b_1 &\geq b - a_1 - a_3 - b_1 \\
c - \tfrac{c}{b}a_1 - \sqrt{\tfrac{c}{b}}a_2 - \sqrt{\tfrac{c}{b}}b_2 &\geq b - a_1 - a_2 - b_2. 
\end{align*}
So the LHS of each inequality in 4b) increases. This implies that all the reduction conditions still hold. Normalizing all vectors to have length 1, we get a proper reduced WR form.
 
\end{proof}

By this lemma, we can restrict our investigation to WR forms, and normalize the WR form $Q$ to make all of the diagonal entries 1. Thus, 
\[
Q = [1,a_1,a_2,a_3;\ 1, b_1,b_2;\ 1,c_1;\ 1]
\] 
and the second-type reduction conditions now read:

\smallskip

\begin{itemize}
\item[4b1)] $0 \leq a_1,a_2,a_3,b_1,b_2,c_1 \leq \frac{1}{2}$.
\item[4b2)] $3+2(a_1+c_1)-2(a_2+b_2+a_3+b_1) \geq 0$, \\
$3+2(a_2+b_2)-2(a_1+c_1+a_3+b_1) \geq 0$, \\
$3+2(a_3+b_1)-2(a_1+c_1+a_2+b_2) \geq 0$.
\end{itemize}

\smallskip

We can see that the six elements $a_1,a_2,a_3,b_1,b_2,c_1$ are indeed equivalent via symmetry. Back to minimizing the solid angle, we fix $\diag(Q) = [1,1,1,1]$. The two Lemmas \ref{tech1} and \ref{tech2} help us find a form $Q'$ with $\diag(Q') = \diag(Q)$ and $\det(Q') \le \det(Q)$. If in addition all entries in $Q'$ are not less than those corresponding in $Q$, then $x^t {Q'}x \geq x^t Qx$ for any $ x\in \sphere$, and so $\Omega_{Q'} \leq \Omega_{Q}$ by equation \eqref{normsolidform}. 
\medskip

We can deduce that when $Q$ has the smallest $\Omega_Q$, we must have $a_1 + c_1 \geq \frac{1}{2}$. For otherwise, by Lemma~\ref{tech1}, the form
\[
Q'=[1,\frac{1}{2}-c_1,\frac{1}{2},\frac{1}{2};\; 1,\frac{1}{2},\frac{1}{2};\; 1,c_1;\; 1]
\]
would have $\Omega_{Q'} < \Omega_{Q}$ according to the previous paragraph's reasoning. By symmetry, we also know that $a_2 + b_2, a_3 + b_1 \geq \frac{1}{2}$. Applying the argument in Theorem~\ref{minonboundary} to the polytope defined by conditions 4b1) and 4b2), we know that at least one condition in $\text{4b}_{1}$) or $\text{4b}_{2}$) must attain equality when $\Omega_Q$ is minimum. Assume it is a condition in $\text{4b}_{2}$), say $3+2(a_1+c_1)-2(a_2+b_2+a_3+b_1) = 0$, since $a_1 + c_1 \geq \frac{1}{2} \geq a_2$ and $3 \geq 2(b_2 + a_3 + b_1)$, it must be that $a_1 + c_1 = \frac{1}{2}$ and $a_2=b_2=a_3=b_1=\frac{1}{2}$. On the other hand, if a condition in $\text{4b}_{1}$) attains equality, we can say it is either $a_1=0$ or $a_1=\frac{1}{2}$. If $a_1=0$, by Theorem~\ref{negisbad} we know that $\Omega_Q \geq \Omega_{Q_0}$, and so it is only necessary to consider when $a_1=\frac{1}{2}$. 
\medskip

All the cases in the above analysis leads to forms with at least one entry equal $\frac{1}{2}$. WLOG, we can assume $a_1 = \frac{1}{2}$, and also $a_2 + b_2, a_3 + b_1 \geq \frac{1}{2}$. If $c_1 > \frac{1}{4}$, Lemma~\ref{tech2}b) implies $Q'=[1,\frac{1}{2},a_2,a_3;\; 1,b_1,b_2;\; 1,\frac{1}{2} ;\; 1]$ has $\Omega_{Q} \geq \Omega_{Q'}$. If $c_1 \leq \frac{1}{4}$, Lemma~\ref{tech2}a) implies $Q''=[1,\frac{1}{2},\frac{1}{2},\frac{1}{2};\; 1,\frac{1}{2},\frac{1}{2};\; 1,c_1;\; 1]$ has  $\Omega_{Q} \geq \Omega_{Q''}$. So now we can restrict to forms with $a_1=c_1=\frac{1}{2}$ ($Q''$ has $a_2=b_2=\frac{1}{2}$, which is equivalent to $a_1=c_1=\frac{1}{2}$ after reordering the basis). By symmetry, we can assume 
\begin{equation}\label{order}
a_3 + b_1 \geq a_2 + b_2 \geq \frac{1}{2}.
\end{equation}
Taking partial derivatives of $\det(Q)$ as a function in $a_3,b_1,a_2,b_2$, we have:
\begin{align*}
&\frac{\partial \det(Q)}{\partial a_3} + \frac{\partial \det(Q)}{\partial b_1} \\
= \;\; &2(a_2 + b_2) - \tfrac{5}{2}(a_3 + b_1) + 2({a_3}{b_1}^2 + {b_1}{a_3}^2) - 2(a_2a_3b_2 + a_2b_1b_2) \\[0.4em]
< \;\; &2(a_2+b_2) - 2(a_3+b_1) - \tfrac{1}{2}(a_3 + b_1) + 2 a_3 b_1 (a_3 + b_1) \\[0.4em]
\leq \;\; &-\tfrac{1}{2}(a_3 + b_1) + 2\tfrac{1}{2}\tfrac{1}{2}(a_3+b_1) = 0.
\end{align*}
Therefore, increasing one of $a_3$ or $b_1$ will decrease $\det(Q)$ and also decrease the solid angle's measure. This can be continued until one of them, say $a_3$, reaches $\frac{1}{2}$. By another application of Lemma~\ref{tech2}, we can further simplify $Q$ so that $b_1=a_3=\frac{1}{2}$. Finally, we have $a_1 = c_1 = a_3 = b_1 = \frac{1}{2}, a_{2} + b_{2} \ge \frac{1}{2}$ and $\Omega_Q$ is now a 2-variable function depending only on $a_2$ and $b_2$. Unfortunately, we cannot reapply the partial derivative argument because of assumption \eqref{order}. The domain for this function is depicted below as the shaded triangular region.

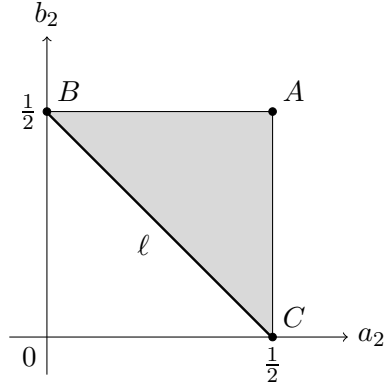
\begin{figure}[!h]
\centering
\begin{tikzpicture}[scale=1]
\draw (0,0) node[below left] {$0$};
\draw[->] (-0.5,0) -- (4,0) node[right] {$a_2$};
\draw[->] (0,-0.5) -- (0,4) node[above] {$b_2$};
\draw[line width=0.5mm] (3,0) -- (0,3);
\draw[fill=gray!30] (3,0) -- (3,3) --(0,3) -- cycle;
\draw[fill=black] (3,3) circle (0.05cm) node[above right] {$A$};
\draw[fill=black] (0,3) circle (0.05cm) node[above right] {$B$} node[left] {$\frac{1}{2}$};
\draw[fill=black] (3,0) circle (0.05cm) node[above right] {$C$} node[below] {$\frac{1}{2}$};
\draw (1.5,1.5) node[below left] {$\ell$};
\end{tikzpicture}
\caption{The reduced domain.}
\label{fig:domain}
\end{figure}

\begin{lem}\label{intlem1}
Consider a reduced WR form:
\[
Q = \left( \begin{smallmatrix} 1 & \frac{1}{2} & a & \frac{1}{2} \\ \frac{1}{2} & 1 & \frac{1}{2} & b \\ a & \frac{1}{2} & 1 & \frac{1}{2} \\ \frac{1}{2} & b & \frac{1}{2} & 1 \end{smallmatrix} \right).
\]
Keep $b$ constant and let $a$ vary in $[\frac{1}{2} - b, \frac{1}{2}]$. The minimum for $\Omega_Q$ occurs at one of the two end points.
\end{lem}
\begin{proof}
First, we have $\det(Q) = a+b-ab-a^{2}-b^{2}+a^{2}b^{2}$. Keeping $b$ constant, we prove that $\Omega_Q$, now considered as a single variable function of $a$, does not have any local minimum in $(\frac{1}{2} - b, \frac{1}{2})$. Calculations will be carried out with $\Omega^2_Q$. Assume that $\Omega^2_Q$ reaches a critical value at point $a$, we have:
\[
\frac{d\Omega^2_Q}{da} = \left( \det \int^2 \right)' = {\det}'{\int^2} + 2\det{\int}{\int}' = {\int}\left({\det}'{\int} + 2\det{\int}' \right) = 0.
\]
Here $\det$ stands for $\det(Q)$ and $\int$ stands for $\int_{\sphere}{(x^{t}Qx)^{-2}}ds$. Thus 
\[
{\det}'{\int} + 2\det{\int}' = 0.
\]
Since $\det$ and $\int$ are positive, $\det'$ and ${\int}'$ have opposite signs. The second derivative of $\Omega^2_Q$ with respect to $a$ is:
\begin{align*}
\frac{d^2 \Omega^2_Q}{d a^2} &= {\det}''{\int}^2 + 4{\det}'{\int}{\int}' + 2\det{\int}'{\int}' + 2\det{\int}{\int}'' \\
&= {\det}''{\int}^2 + 3{\det}'{\int}{\int}' + {\int}'\left( {\det}'{\int} + 2\det{\int}' \right) + 2\det{\int}{\int}'' \\
&= {\int} \left( {\det}''{\int} + 3{\det}'{\int}' + 2\det{\int}'' \right).
\end{align*}

Since $\det'$ and ${\int}'$ have opposite signs, the term $3 \det'{\int}'$ is negative. If we have $\det''{\int} + 2\det{\int}'' < 0$, then $\frac{d^2 \Omega^2_Q}{d a^2} < 0$, which means $a$ cannot be a local minimum. We show this is indeed the case. Note that $\det$ is a polynomial in $a$ with degree $2$, and $\det'' = -2(1-{b}^2)$. Note that $(1-{b}^2)$ is the determinant of $\left( \begin{smallmatrix} 1 & b \\ b & 1 \end{smallmatrix} \right)$. So $(1-{b}^2)$ is the squared area of the parallelogram formed by the two vectors $v_2$ and $v_4$. This parallelogram is in turn a 2-dimensional face of the 4-dimensional parallelepiped formed by $v_1,v_2,v_3,v_4$. Since all the four vectors have length $1$, the volume of this parallelepiped is less than or equal to the area of the parallelogram. Since $\det$ is the squared volume of the parallelepiped, this results in $-\det'' = 2(1-{b}^2) \geq 2\det$. Now it remains to prove $\int > {\int}''$. We have:
\begin{align*}
\int &=  \int_{\sphere}\frac{ds}{(1 + x_1 x_2 + x_2 x_3 + x_3 x_4 + x_1 x_4 + 2 a x_1 x_3 + 2 b x_2 x_4 )^2} \\
&\geq \int_{\sphere}\frac{ds}{(1 + x_1 x_2 + x_2 x_3 + x_3 x_4 + x_1 x_4 + x_1 x_3 + x_2 x_4 )^2} \\
&\approx 0.345503\dots
\end{align*}
and
\begin{align*}
{\int}'' &= 6\int_{\sphere}\frac{4{x_1}^2 {x_3}^2ds}{(1 + x_1 x_2 + x_2 x_3 + x_3 x_4 + x_1 x_4 + 2 a x_1 x_3 + 2 b x_2 x_4 )^4}\\
&\leq 6\int_{\sphere}\frac{4{x_1}^2 {x_3}^2ds}{(1 + x_1 x_2 + x_2 x_3 + x_3 x_4 + x_1 x_4)^4} \\
&\approx 0.215663\dots
\end{align*}
Here we used differentiation through the integral sign to compute ${\int}''$. 

\end{proof}

The previous lemma is also applicable if we consider $\Omega_Q$ as a function of $b_2$ with $a_2$ being fixed. Hence, it tells us that the minimum for $\Omega_Q$ must occur either on the segment $\ell$ or at the point $A$ in Figure~\ref{fig:domain}. The next lemma ensures that $\Omega_Q$ takes smaller values at $B$ and $C$ compared to other points on $\ell$. Thus, over all, the minimum for $\Omega_Q$ should be either at $A$ or $B$ and $C$, i.e., either $\Omega_{Q_{\A_4}}$ or $\Omega_{Q_0}$.

\begin{lem}\label{intlem2}
Consider a reduced WR form:
\[
Q = \left( \begin{smallmatrix} 1 & \frac{1}{2} & a & \frac{1}{2} \\ \frac{1}{2} & 1 & \frac{1}{2} & \left(\frac{1}{2}-a\right) \\ a & \frac{1}{2} & 1 & \frac{1}{2} \\ \frac{1}{2} & (\frac{1}{2} - a) & \frac{1}{2} & 1 \end{smallmatrix} \right).
\]
Let $a$ vary in $[0,\frac{1}{2}]$, the minimum for $\Omega_Q$ occurs at the two end points.
\end{lem}
\begin{proof}
We argue that $\Omega^2_Q$ has no local minimum $a \in (0,\frac{1}{2})$. Assume $\frac{d \Omega^2_Q}{da}=0$, we show that $\frac{d^2 \Omega^2_Q}{da^2} < 0$. As before, we would need to prove $-\det(Q)''\int > 2\det(Q) \int''$. In this case:
\begin{align*}
\det(Q) &= a^4 - a^3 - \tfrac{3a^2}{4} + \tfrac{a}{2} + \tfrac{1}{4} \\[0.4em]
&= (a-\tfrac{1}{4})^4 - \tfrac{9}{8}(a-\tfrac{1}{4})^2 + \tfrac{81}{256},
\end{align*}
and
\begin{align*}
\det(Q)'' &= 12(a-\tfrac{1}{4})^2 - \tfrac{9}{4}.\qquad\qquad\;\;\;\;
\end{align*}
It is easy to check that $-\det''(Q) \geq 6\det(Q)$:
\begin{align*}
-\det(Q)'' - 6\det(Q) &= \tfrac{9}{4} - 12(a-\tfrac{1}{4})^2 - 6((a-\tfrac{1}{4})^4 - \tfrac{9}{8}(a-\tfrac{1}{4})^2 + \tfrac{81}{256})\\[0.4em]
&= \tfrac{45}{128} - 6(a-\tfrac{1}{4})^4 - \tfrac{21}{4}(a-\tfrac{1}{4})^2\\[0.4em]
&\geq \tfrac{45}{128} - 6(\tfrac{1}{4})^4 - \tfrac{21}{4}(\tfrac{1}{4})^2 = 0.
\end{align*}
And thus it remains to prove $\int > \frac{1}{3}{\int}''$:
\begin{align*}
\int &= \int_{\sphere}\frac{ds}{(1 + x_1 x_2 + x_2 x_3 + x_3 x_4 + x_1 x_4 + 2 a x_1 x_3 + 2 (\frac{1}{2}-a)x_2 x_4 )^2}\\
&\geq \int_{\sphere}\frac{ds}{(1 + x_1 x_2 + x_2 x_3 + x_3 x_4 + x_1 x_4 + x_1 x_3 + x_2 x_4 )^2} \\
&\approx 0.345503
\end{align*}
and
\begin{align*}
\frac{1}{3}{\int}'' &= 2\int_{\sphere}\frac{4(x_1 x_3 - x_2 x_4)^2ds}{(1 + x_1 x_2 + x_2 x_3 + x_3 x_4 + x_1 x_4 + 2 a x_1 x_3 + 2 (\frac{1}{2}-a) x_2 x_4 )^4} \\
&\leq 2\int_{\sphere}\frac{4(x_1 x_3 - x_2 x_4)^2 ds}{(1 + x_1 x_2 + x_2 x_3 + x_3 x_4 + x_1 x_4)^4} \\
&\approx 0.0773524
\end{align*}
\end{proof}

\begin{theo}
Any rank $4$ lattice has a reduced basis with $\Omega_{Q_{\A_4}} \leq \Omega_Q \leq \frac{S_3}{16} = \frac{\pi^2}{8}$.
\end{theo}
\begin{proof}
Take a reduced basis for $\lattice$. We can switch some basis vectors to their negatives to make sure that $\bar{\Omega}_Q \leq \frac{1}{16}$. This gives the upper bound. By the previous two lemmas, the lower bound is certain if we can show $\Omega_{Q_{\A_4}} < \Omega_{Q_0}$. This will be verified numerically in the next section, together with a 5-dimensional example.

\end{proof}

\medskip

\section{A counter-example in $\R^5$ and some discussion}\label{Secdim5}

In order to finish Theorem 4.5, we need to compare $\Omega_{Q_{\A_4}}$ and $\Omega_{Q_0}$ numerically. Besides, we also have to verify the integral values in Lemma~\ref{intlem1} and \ref{intlem2}. Using spherical coordinates in $\R^4$, we can take:
\begin{align*}
x_1 &= \cos \alpha\\
x_2 &= \sin \alpha \cos \beta\\
x_3 &= \sin \alpha \sin \beta \cos \gamma\\
x_4 &= \sin \alpha \sin \beta \sin \gamma
\end{align*}
with $0 \leq \alpha,\beta,\gamma \leq \frac{\pi}{2}$ and the jacobian $ds = \sin^2 \alpha\sin\beta \; d\alpha d\beta d\gamma$. We implemented this trigonometric parametrization with \textit{MATHEMATICA} to compute the values of the integrals in Lemma~\ref{intlem1} and \ref{intlem2} and also verified that:
\[
\Omega_{Q_{\A_4}} = 0.193142\ldots \; < \; \Omega_{Q_0} = 0.205617\ldots
\]

Interestingly, the situation reverses in $\R^5$ with the analogues of $Q_{\A_4}$ and $Q_0$. Take
\[
R_{\A_5} = \left( \begin{matrix} 1&\frac{1}{2}&\frac{1}{2}&\frac{1}{2}&\frac{1}{2}\\[0.3em] \frac{1}{2}&1&\frac{1}{2}&\frac{1}{2}&\frac{1}{2}\\[0.3em] \frac{1}{2}&\frac{1}{2}&1&\frac{1}{2}&\frac{1}{2}\\[0.3em] \frac{1}{2}&\frac{1}{2}&\frac{1}{2}&1&\frac{1}{2}\\[0.3em] \frac{1}{2}&\frac{1}{2}&\frac{1}{2}&\frac{1}{2}&1 \end{matrix} \right) \qquad \text{and} \qquad R_0 = \left( \begin{matrix} 1&0&\frac{1}{2}&\frac{1}{2}&\frac{1}{2}\\[0.3em] 0&1&\frac{1}{2}&\frac{1}{2}&\frac{1}{2}\\[0.3em] \frac{1}{2}&\frac{1}{2}&1&\frac{1}{2}&\frac{1}{2}\\[0.3em] \frac{1}{2}&\frac{1}{2}&\frac{1}{2}&1&\frac{1}{2}\\[0.3em] \frac{1}{2}&\frac{1}{2}&\frac{1}{2}&\frac{1}{2}&1 \end{matrix} \right),
\]
then both forms are reduced and: 
$$\Omega_{\A_5} = 0.0505862\ldots \; > \; \Omega_{R_0} = 0.0479361\ldots$$
Thus, among all reduced forms, those of the alternating lattices produce the smallest solid angles for dimensions less than 5 but not higher. Similar to $Q_0$, $R_0$ is known to have the smallest determinant among all WR forms.
\medskip

Let us also briefly discuss the intuition behind Lemma~\ref{intlem1} and \ref{intlem2}. The absence of local minima for $\Omega_Q$, considered as a univariate function, within the open interval can be interpreted as $\Omega_Q$ being locally quasi-concave. To a somewhat greater extent, the method employed in these two lemmas are also adequate to prove quasi-concavity for a univariate $\Omega_Q$, without assuming that $Q$ is WR or reduced. If we consider $\Omega_Q$ as a multivariate function however, naive differentiation does not seem enough to establish global quasi-concavity. Such a result, if settled, may shed some light on the behavior of \textit{spherical volumes} in higher dimensional spherical geometry. 
\medskip

Lemma~\ref{WRgood} was also an important step in our proof. It essentially says that any rank 4 reduced form $Q$ can be normalized to a WR form and still remains reduced. We wonder if such a similar result still holds in higher dimension
\medskip

\begin{ques}
Is the minimum for $\Omega_Q$ always located among WR forms in any dimension?
\end{ques}

Lastly, we would like to revisit Corollary~\ref{minor}, which says that the solid angle does not exceed $\frac{1}{2^n}$ for any basis with non-obtuse pairwise angles. One can ask a more direct question: 

\begin{ques}
Is it always possible to completely embed any such basis into the positive orthant of $\R^n$?
\end{ques}

By embedding we mean simultaneously rotating all the basis vectors with an orthogonal transformation. Geometric intuition tells us the affirmative, which is obvious up to at least dimension 3. Fortunately, the full answer is known: YES, such an embedding exists if $n < 5$, but NO in general. The interested reader can find the detailed answer in the sizable literature written on this topic (ref.~\cite{BerMon}). It is interesting to see how familiar intuitions can break down when we go up in dimensions.

\end{document}